\documentclass[12pt]{article}
\usepackage{a4}
\usepackage{amsmath}
\usepackage{amsthm}
\usepackage{amsfonts}
\usepackage{amssymb}
\usepackage{cite}
\usepackage{epsfig}
\newtheorem{theorem}{Theorem}
\newtheorem{lemma}[theorem]{Lemma}
\newtheorem{problem}{Problem}
\def\NN{{\mathbb N}}
\def\ss{\boldsymbol s}
\def\dd{\;{\rm d}}
\def\EF{Ehrenfeucht-Fra\"\i{}ss\'e}
\def\opsi{{\widehat{\psi}}}
\begin{document}
\title{First order convergence and roots\thanks{This work was done during a visit to the Institut Mittag-Leffler (Djursholm, Sweden).}}
\author{Demetres Christofides\thanks{School of Sciences, UCLan Cyprus, 7080 Pyla, Cyprus. E-mail: {\tt dchristofides@uclan.ac.uk}.}\and
        Daniel Kr\'al'\thanks{Mathematics Institute, DIMAP and Department of Computer Science, University of Warwick, Coventry CV4 7AL, UK. E-mail: {\tt d.kral@warwick.ac.uk}. The work leading to this invention has received funding from the European Research Council under the European Union's Seventh Framework Programme (FP7/2007-2013)/ERC grant agreement no.~259385.}
	}
\date{}
\maketitle
\begin{abstract}
Ne\v set\v ril and Ossona de Mendez introduced the notion of first order convergence,
which unifies the notions of convergence for sparse and dense graphs. They asked
whether if $(G_i)_{i\in\NN}$ is a sequence of graphs with $M$ being
their first order limit and $v$ is a vertex of $M$,
then there exists a sequence $(v_i)_{i\in\NN}$ of vertices such that
the graphs $G_i$ rooted at $v_i$ converge to $M$ rooted at $v$.
We show that this holds for almost all vertices $v$ of $M$ and
we give an example showing that the statement need not hold for all vertices.
\end{abstract}

\section{Introduction}
\label{sec-intro}

The theory of limits of combinatorial objects keeps attracting more and more attention and
its applications in various areas such as extremal combinatorics, computer science and many others grow.
The most understood is the case of dense graph convergence which originated in the series of papers
by Borgs, Chayes, Lov\'asz, S\'os, Szegedy and Vesztergombi~\cite{bib-borgs08+,bib-borgs+,bib-borgs06+,bib-lovasz06+,bib-lovasz10+}.
This development is also reflected in a recent monograph by Lov\'asz~\cite{bib-lovasz-book}.
Another line of research concentrated around the convergence of sparse graphs (such as those
with bounded maximum degree) known as the Benjamini-Schramm convergence~\cite{bib-aldous07+,bib-benjamini01+,bib-elek07,bib-hatami+}.
Ne\v set\v ril and Ossona de Mendez~\cite{bib-folim1,bib-folim2} proposed a notion of
first order convergence to unify the two notions for the dense and sparse settings.

First order convergence is a notion of convergence for all relational structures.
For simplicity, we limit our exposition to graphs and rooted graphs only
but all our arguments extend to the general setting naturally.
If $\psi$ is a first order formula with $k$ free variables and $G$ is a finite graph,
then the {\em Stone pairing} $\langle \psi,G\rangle$ is the probability that
a uniformly chosen $k$-tuple of vertices of $G$ satisfies $\psi$.
A sequence $(G_n)_{n\in\NN}$ of graphs is called {\em first order convergent}
if the limit $\lim\limits_{n\to\infty}\langle \psi,G_n\rangle$ exists for every first order formula $\psi$.
A {\em modeling} $M$ is a (finite or infinite) graph whose vertex set is equipped with a probability measure
such that the set of all $k$-tuples of vertices of $M$ satisfying a formula $\psi$ is measurable
in the product measure for every first order formula $\psi$ with $k$ free variables.
In the analogy to the graph case, the {\em Stone pairing} $\langle \psi,M\rangle$
is the probability that a randomly chosen $k$-tuple of vertices satisfies $\psi$.
If a finite graph is viewed as a modeling with a uniform discrete probability measure on its vertex set,
then the stone pairings for the graph and the modeling obtained in this way coincide.

A modeling $M$ is a limit of a first order convergent sequence $(G_n)_{n\in\NN}$
if $$\lim_{n\to\infty}\langle \psi,G_n\rangle=\langle\psi,M\rangle$$
for every first order formula $\psi$.
It is not true that every first order convergent sequence of graphs has a limit modeling~\cite{bib-folim2}
but it can be shown, e.g., that first order convergent sequences of trees do~\cite{bib-folim-trees,bib-folim3}.

Ne\v set\v ril and Ossona de Mendez~\cite[Problem 1]{bib-folim2} raised the following problem,
which we formulate here for graphs only.

\begin{problem}
\label{prob-1}
Let $M$ be a modeling that is a limit of a first order convergent sequence $(G_n)_{n\in\NN}$ and
let $v$ be a vertex of $M$. Does there exist a sequence $(v_n)_{n\in\NN}$ of vertices of the graphs $(G_n)_{n\in\NN}$
such that the modeling $M$ rooted at $v$ is a limit of the sequence $(G'_n)_{n\in\NN}$
where $G'_n$ is obtained from $G_n$ by rooting at $v_n$?
\end{problem}

We prove that the statement from Problem~\ref{prob-1} is true for almost every vertex $v$ of $M$.

\begin{theorem}
\label{thm-main}
Let $M$ be a modeling that is a limit of a first order convergent sequence $(G_n)_{n\in\NN}$.
It holds with probability one that if $M'$ is a modeling obtained from $M$ by rooting at a random vertex $v$ of $M$,
then there exist a sequence $(v_n)_{n\in\NN}$ of vertices of $(G_n)_{n\in\NN}$
such that $M'$ is a limit of the sequence $(G'_n)_{n\in\NN}$
where $G'_n$ is obtained from $G_n$ by rooting at $v_n$.
\end{theorem}

Theorem~\ref{thm-main} follows from a more general Theorem~\ref{thm-roots}
which we prove in Section~\ref{sec-roots}. In Section~\ref{sec-counter},
we present an example that the statement of Theorem~\ref{thm-main} cannot
be strengthened to all vertices $v$ of $M$, i.e., the answer to Problem~\ref{prob-1}
is negative. This also answers a more general problem \cite[Problem 2]{bib-folim2}
in the negative way.

\section{Notation}
\label{sec-notation}

We assume that the reader is familiar with standard graph theory and logic terminology as
it can be found, e.g., in~\cite{bib-diestel,bib-ebbinghaus+}. We briefly review here
less standard terminology and notation only. Throughout the paper,
we write $[k]$ for the set of positive integers between $1$ and $k$ (inclusively).

There is a close connection between the first order logic and the so-called \EF{} games.
The {\em $p$-round \EF{} game} is played by two players, the spoiler and the duplicator,
on two relational structures. We explain the game when played on two graphs $G$ and $H$.
At the beginning of each round, the spoiler chooses a vertex in any one of the two graphs and
the duplicator responds with choosing a vertex in the other.
One vertex can be chosen several times in different rounds of the game.
Let $v_i$ be the vertex chosen in the $i$-th round in $G$ and $w_i$ the vertex chosen in the $i$-th round in $H$.
The duplicator wins the game if the subgraph of $G$ induced by $v_1,\ldots,v_p$ and
the subgraph of $H$ induced by $w_1,\ldots,w_p$ are isomorphic through the isomorphism
mapping $v_i$ to $w_i$.

It can be shown that the duplicator has a winning strategy for the $p$-round \EF{} game played on $G$ and $H$
if and only if $G$ and $H$ satisfy the same first order sentences with quantifier depth at most $p$.
More generally,
suppose that
$\psi(x_1,\ldots,x_k)$ is a first order formula with $k$ free variables and with quantifier depth $d$,
$G$ and $H$ are two graphs, and $v_1,\ldots,v_k$ and $w_1,\ldots,w_k$ are (not necessarily distinct) vertices of $G$ and $H$, respectively.
If the duplicator has a winning strategy for the $(k+d)$-round \EF{} game when played on $G$ and $H$
with the vertices $v_1,\ldots,v_k$ and $w_1,\ldots,w_k$ played in the first $k$ rounds (so, it remains to play $d$ rounds of the game),
then $G$ satisfies $\psi(v_1,\ldots,v_k)$ if and only if $H$ satisfies $\psi(w_1,\ldots,w_k)$.
This correspondence can be used to show~\cite{bib-ebbinghaus+} that
the set $F^m_{p,q}$ of all non-equivalent first order formulas with $p$ free variables and quantifier depth at most $q$
for $m$-rooted graphs
is finite for all positive integers $m$, $p$ and $q$ (the language for $m$-rooted graphs consists of a single binary
relation representing the adjacency and $m$ constants representing the roots).

\section{Almost every rooting is good}
\label{sec-roots}

In this section, we prove our main result
which provides a positive answer to Problem~\ref{prob-1} in the almost every sense.
As preparation for the proof, we need to establish several technical lemmas.

\begin{lemma}
\label{lm-formula}
Let $\psi$ be a first order formula for $m$-rooted graphs and let $[a,b]\subseteq [0,1]$ be a non-empty interval.
For every $\varepsilon>0$, there exists a first order formula $\psi'$ such that
the following holds for every $m$-rooted modeling $M$:
\begin{itemize}
\item if $\langle\psi,M\rangle\in [a,b]$, then $\langle\psi',M\rangle>1-\varepsilon$, and
\item if $\langle\psi,M\rangle\not\in (a-\varepsilon,b+\varepsilon)$, then $\langle\psi',M\rangle<\varepsilon$.
\end{itemize}
\end{lemma}

\begin{proof}
If $\psi$ is a sentence, i.e., it has no free variables, then the statement is trivial.
Suppose that $\psi$ has $k$ free variables.
Let $\psi_n$ be the first order formula with $nk$ free variables grouped in $n$ $k$-tuples such that
$\psi_n$ is true if and only if at least $an-n^{2/3}$ and at most $bn+n^{2/3}$ of these $k$-tuples
do satisfy $\psi$. Formally,
$$\begin{array}{cl}
  &\psi_n(x^1_1,\ldots,x^1_k,\ldots,x^n_1,\ldots,x^n_k)\\
  =&\bigvee\limits_{i=\lceil an-n^{2/3}\rceil}^{\lfloor bn+n^{2/3}\rfloor}
                                                                          \bigvee\limits_{A\in {[n]\choose i}}
	 \left(\bigwedge\limits_{j\in A}\psi(x^j_1,\ldots,x^j_k)\land\bigwedge\limits_{j\not\in A}\neg\psi(x^j_1,\ldots,x^j_k)\right)							\;\mbox{.}
	\end{array}$$
The Chernoff bound implies that the formula $\psi'$ can be chosen to be the formula $\psi_n$ for $n$ sufficiently large.
\end{proof}

An interval is a {\em dyadic interval of order $k\in\NN$} if it is of the form $[a2^{-k},(a+1)2^{-k}]$ for some integer $a$.
A point $x$ is {\em $\varepsilon$-far} from an interval $J$ if $|x-y|\ge\varepsilon$ for every $y\in J$.
Otherwise, we say that $x$ is {\em $\varepsilon$-close} to $J$.
A {\em multidimensional interval} is a subset of $[0,1]^d$ that is the product of $d$ intervals;
if $J$ is a multidimensional interval, then $J_i$ denotes the $i$-th term in the product.
A multidimensional interval $J$ is {\em dyadic of order $k\in\NN$} if every $J_i$ is dyadic of order $k$.

The next lemma is a direct consequence of Lemma~\ref{lm-formula}.
Recall that $F^m_{p,q}$ is the set of all non-equivalent first order formulas with $p$ free variables and quantifier depth at most $q$, and
the set $F^m_{p,q}$ is finite for all $m$, $p$ and $q$.

\begin{lemma}
\label{lm-formulas}
Let $m$, $p$ and $q$ be integers and let $J\subseteq [0,1]^{F^m_{p,q}}$ be a multidimensional interval.
For every $\varepsilon>0$, there exists a first order formula $\psi^{m,J,\varepsilon}_{p,q}$ such that
the following holds for every $m$-rooted modeling $M$:
\begin{itemize}
\item if $\langle\psi,M\rangle\in J_\psi$ for every $\psi\in F^m_{p,q}$, then $\langle\psi^{m,J,\varepsilon}_{p,q},M\rangle>1-\varepsilon$, and
\item if $\langle\psi,M\rangle$ is $\varepsilon$-far from $J_\psi$ for at least one $\psi\in F^m_{p,q}$, then $\langle\psi^{m,J,\varepsilon}_{p,q},M\rangle<\varepsilon$.
\end{itemize}
\end{lemma}

If $\psi^{m,J,\varepsilon}_{p,q}$ is the formula from Lemma~\ref{lm-formulas},
then $\opsi^{m,J,\varepsilon}_{p,q}$ is the formula obtained from $\psi^{m,J,\varepsilon}_{p,q}$ by adding $m$ new free variables such that
$\opsi^{m,J,\varepsilon}_{p,q}$ is satisfied if and only if $\psi^{m,J,\varepsilon}_{p,q}$ is satisfied for the modeling obtained from $M$
by rooting at the $m$-tuple specified by the new free variables, i.e., the $m$ constants in $\psi^{m,J,\varepsilon}_{p,q}$ are replaced
with the new $m$ free variables of $\psi^{m,J,\varepsilon}_{p,q}$.
An $m$-tuple of vertices $v_1,\ldots,v_m$ of a modeling $M$ is {\em negligible}
if there exist integers $p$ and $q$ and a dyadic multidimensional interval $J\subseteq [0,1]^{F^m_{p,q}}$ such that
\begin{itemize}
\item $\langle\psi,M'\rangle_{\psi\in F^m_{p,q}}\in J$ where $M'$ is the $m$-rooted modeling obtained from $M$ by rooting at $v_1,\ldots,v_m$, and
\item there exists $\varepsilon_0>0$ such that $\langle\opsi^{m,J,\varepsilon}_{p,q},M\rangle\le\varepsilon$ for every $0<\varepsilon<\varepsilon_0$.
\end{itemize}
The next lemma asserts that very few $m$-tuples can be negligible.

\begin{lemma}
\label{lm-negligible}
If $M$ is a modeling and $m$ is an integer, then the set of negligible $m$-tuples of $M$ is a subset of a set of measure zero.
\end{lemma}

\begin{proof}
Note that there are countably many triples $p$, $q$ and $J\subseteq [0,1]^{F^m_{p,q}}$ where $J$ is dyadic.
Hence,
it is enough to show for every $p$, $q$ and $J$,
that
if there exists $\varepsilon_0>0$ such that $\langle\opsi^{m,J,\varepsilon}_{p,q},M\rangle<\varepsilon$ for every $0<\varepsilon<\varepsilon_0$,
then there exists a set of measure zero containing all $m$-tuples $v_1,\ldots,v_m$ such that $\langle\psi,M'\rangle_{\psi\in F^m_{p,q}}\in J$ where $M'$ is obtained from $M$ by rooting at $v_1,\ldots,v_m$.
Fix $p$, $q$ and $J$ for the rest of the proof.
Let $X$ be the set of all such $m$-tuples, and
let $k_0$ be an integer such that $2^{-k_0}<\varepsilon_0$.

Let $F_k(v_1,\ldots,v_m)$ for $k\in\NN$ be the function from $M^m$ to $[0,1]$ defined to be $\langle\psi^{m,J,2^{-k}}_{p,q},M'\rangle$
where $M'$ is the modeling obtained from $M$ by rooting at $v_1,\ldots,v_m$.
Since the set of tuples satisfying $\opsi^{m,J,2^{-k}}_{p,q}$ is measurable,
the function $F_k$ is measurable in the corresponding product space.
Moreover, it holds that
$$\int F_k(v_1,\ldots,v_m)\dd v_1\cdots v_m=\langle\opsi^{m,J,2^{-k}}_{p,q},M\rangle<2^{-k}$$
for every $k\ge k_0$.
Observe that Lemma~\ref{lm-formulas} yields that
$$X\subseteq\bigcap_{k=k_0}^{\infty}F_k^{-1}([1-2^{-k},1])\;\mbox{.}$$
Since the function $F_k$ takes values between $0$ and $1$ (inclusively),
the measure of $F_k^{-1}([1-2^{-k},1])$ is less than $2^{-k}/(1-2^{-k})$.
It follows that $X$ is a subset of a set of measure zero.
\end{proof}

We are now ready to prove our main theorem.

\begin{theorem}
\label{thm-roots}
Let $M$ be a modeling that is a limit of a first order convergent sequence $(G_n)_{n\in\NN}$ and
let $m$ be a positive integer.
It holds with probability one that if $M'$ is a modeling obtained from $M$ by rooting at a random $m$-tuple of vertices of $M$,
then there exist a sequence $(v_{n,1},\ldots,v_{n,m})_{n\in\NN}$ of $m$-tuples such that
the graphs $(G_n)_{n\in\NN}$ rooted at these $m$-tuples first order converge to $M'$.
\end{theorem}

\begin{proof}
By Lemma~\ref{lm-negligible},
we can assume that the randomly chosen $m$-tuple $w_1,\ldots,w_m$ of the vertices of $M$ is not negligible.
It is enough to show for every $p$, $q$ and $\delta>0$ that there exists $n_0$ such that
every graph $G_n$, $n\ge n_0$, contains an $m$-tuple $v_{n,1},\ldots,v_{n,m}$ of vertices such that
the graph $G'_n$ obtained from $G_n$ by rooting at the vertices $v_{n,1},\ldots,v_{n,m}$ satisfies that
\begin{equation}
|\langle\psi,G'_n\rangle-\langle\psi,M'\rangle|\le\delta\label{eq-1}
\end{equation}
for every $\psi\in F^m_{p,q}$. Note that (\ref{eq-1}) implies that
$$|\langle\psi,G'_n\rangle-\langle\psi,M'\rangle|\le\delta$$
for every $\psi\in F^m_{p',q'}$ for $p'\in [p]$ and $q'\in [q]$.

Fix the integers $p$ and $q$ and the real $\delta>0$ for the rest of the proof.
Choose an integer $k$ such that $2^{-k}<\delta$ and a real $\varepsilon>0$ such that $2^{-k}+\varepsilon<\delta$.
Further, let $J\subseteq [0,1]^{F^m_{p,q}}$ be the dyadic multidimensional interval of order $k$
containing the point $\langle\psi,M'\rangle_{\psi\in F^m_{p,q}}$.
Since the $m$-tuple $v_{n,1},\ldots,v_{n,m}$ is not negligible, there exists $\varepsilon'<\varepsilon$ such that
$$\langle\opsi^{m,J,\varepsilon'}_{p,q},M\rangle>\varepsilon'\;\mbox{.}$$
Since the sequence $(G_n)_{n\in\NN}$ converges to $M$, there exists $n_0$ such that
\begin{equation}
\langle\opsi^{m,J,\varepsilon'}_{p,q},G_n\rangle>\varepsilon'\label{eq-2}
\end{equation}
for every $n\ge n_0$.
By the definition of the formula $\opsi^{m,J,\varepsilon'}_{p,q}$,
the inequality (\ref{eq-2}) implies that
every graph $G_n$, $n\ge n_0$, contains an $m$-tuple $v_{n,1},\ldots,v_{n,m}$ of vertices such that
\begin{equation}
\langle\psi^{m,J,\varepsilon'}_{p,q},G'_n\rangle>\varepsilon'\label{eq-3}
\end{equation}
where $G'_n$ is obtained from $G_n$ by rooting at $v_{n,1},\ldots,v_{n,m}$.
By Lemma~\ref{lm-formulas},
the Stone pairing $\langle\psi,G'_n\rangle$ is $\varepsilon'$-close to $J_\psi$ for every $\psi\in F^m_{p,q}$.
It follows that
$$|\langle\psi,G'_n\rangle-\langle\psi,M\rangle|<2^{-k}+\varepsilon'<\delta$$
for every $\psi\in F^m_{p,q}$.
The proof of the theorem is now finished.
\end{proof}

\section{Counterexample}
\label{sec-counter}

We now show that the statement of Theorem~\ref{thm-main} cannot be strengthened to all vertices.
Before doing so, we need to introduce some additional notation.

If a (finite or infinite) graph $G$ is bipartite,
we write $G(A,B)$ where $A$ and $B$ are the two parts of $G$.
The adjacency matrix $M$ of $G$ is the matrix with rows indexed by $A$ and columns indexed by $B$ such that
$M_{ab}$ is equal to $1$ if the vertices $a\in A$ and $b\in B$ are adjacent, and it is equal to zero, otherwise.
If $G(A,B)$ is a bipartite graph and $W$ is a subset of its vertices,
then $W_A$ is $A\cap W$ and $W_B$ is $B\cap W$.
The adjacency matrix of $G$ restricted to $W$ is the submatrix with rows and columns indexed by $W_A$ and $W_B$, respectively.
Suppose that $W$ is a subset of vertices of $G(A,B)$, $W'$ is a subset of vertices of $G'(A',B')$ and
there is a one-to-one correspondence between the vertices of $W$ and $W'$.
When we say that the adjacency matrices of $G$ and $G'$ restricted to $W$ and $W'$ are the same,
we mean that they are the same in the stronger sense that the rows/columns for the corresponding vertices are the same.

A bipartite graph $G(A,B)$ is $\ell$-universal,
if every vector from $\{0,1\}^B$ appears at least $\ell$ times among the rows of the adjacency matrix of $G$.
If $W$ is a subset of vertices of $G(A,B)$,
then the $\ell$-shadow of $W$ is the multiset $S$ such that
each of the vectors $u\in\{0,1\}^{W_A}$ is included to $S$ exactly $\min\{k,\ell\}$ times
where $k$ it the number of times $u$ appears among the columns of the adjacency matrix of $G$ restricted to $W_A\times (B\setminus W_B)$.
If $W_A=\emptyset$, then the $\ell$-shadow of $W$ consists of $\min\{|B|,\ell\}$ null vectors (i.e. vectors of dimension zero).

The following is the key lemma in our construction.

\begin{lemma}
\label{lm-key}
Let $p$ and $q$ be two non-negative integers.
Suppose that $G(A,B)$ and $G'(A',B')$ are two $(p+q)$-universal graphs and
that $w_1,\ldots,w_q$ and $w'_1,\ldots,w'_q$ are two sequences of the vertices of $G(A,B)$ and $G'(A',B')$, respectively.
Let $W=\{w_1,\ldots,w_q\}$ and $W'=\{w'_1,\ldots,w'_q\}$.
If the adjacency matrices of $G(A,B)$ and $G'(A',B')$ restricted to $W_A\times W_B$ and to $W'_{A'}\times W'_{B'}$, respectively,
are the same (with the row/column corresponding to $w_i$ being the same as that of $w'_i$), and
the $2^{p}$-shadows of $W$ and $W'$ are also the same, then the duplicator has a winning strategy
for the $(p+q)$-round \EF{} game where the vertices chosen in the first $q$ rounds are $w_1,\ldots,w_q$ and $w'_1,\ldots,w'_q$.
\end{lemma}

\begin{proof}
We proceed by induction on $p$.
If $p=0$, then the graphs induced by the vertices of $W$ and $W'$ are isomorphic
since the adjacency matrices of $G$ and $G'$ restricted to $W$ and $W'$ are the same.

Suppose that $p>0$.
By symmetry, we can assume that the spoiler chooses a vertex of $G$ in the next round.
Let $w_{q+1}$ be the chosen vertex.
If $w_{q+1}=w_i$ for some $1\le i\le q$, the duplicator responds with $w'_i$.
So, we can now assume that $w_{q+1}$ is different from all the vertices $w_1,\ldots,w_q$ and
we distinguish two cases based on whether $w_{q+1}$ belongs to $A$ or $B$.

Let us start with the analysis of the case when $w_{q+1}\in A\setminus A_W$.
Let $x$ be the row of the adjacency matrix of $G$ corresponding to $w_{q+1}$.
We will construct a vector $x'\in\{0,1\}^{B'}$ which will determine the response of the duplicator.

Set $x'_{w'_i}=x_{w_i}$ for $w'_i\in W'_{B'}$.
Fix a vector $u\in\{0,1\}^{W_A}$.
Let $u_0,u_1\in\{0,1\}^{W_A\cup\{w_{q+1}\}}$ be the two extensions of $u$, and
let $m_0$ and $m_1$ be the multiplicities of $u_0$ and $u_1$, respectively, in the $2^{p-1}$-shadow of $W\cup\{w_{q+1}\}$.
Finally, let $W_u$ be the set of the vertices $v$ of $B'\setminus W'_{B'}$ such that the column of $v$ restricted to $W'_{A'}$ is $u$.
If $m_0+m_1<2^{p}$, then the $2^p$-shadow of $W'$ contains the vector $u$ exactly $m_0+m_1$ times.
Set $x'_v$ to $0$ for $m_0$ of the vertices $v\in W_u$ and to $1$ for $m_1$ of such vertices.
If $m_0+m_1\ge 2^p$, at least one of the numbers $m_0$ or $m_1$ is at least $2^{p-1}$.
If $m_0\ge 2^{p-1}$, set $x'_v$ to $1$ for $\min\{m_1,2^{p-1}\}$ of vertices $v\in W_u$ and to $0$ for all other $v\in W_u$.
If $m_0<2^{p-1}$ and $m_1\ge 2^{p-1}$, set $x'_v$ to $0$ for $m_0$ of vertices $v\in W_u$ and to $1$ for all other $v\in W_u$.
Performing this for every vector $u\in\{0,1\}^{W_A}$, the entire vector $x\in\{0,1\}^{B'}$ is defined.

Since the graph $G'$ is $(p+q)$-universal,
there exists a vertex $w'_{q+1}\in A'$ different from the vertices $w'_1,\ldots,w'_q$
such that the row of $w'_{q+1}$ in the adjacency matrix of $G'$ is equal to $x'$.
The duplicator responds with the vertex $w'_{q+1}$.
Observe that the choice of $x'$ implies that
the adjacency matrices of $G$ and $G'$ restricted to $W\cup\{w_{q+1}\}$ and $W'\cup\{w'_{q+1}\}$, respectively, are the same and
that the $2^{p-1}$-shadows of $W\cup\{w_{q+1}\}$ and $W'\cup\{w'_{q+1}\}$ are also the same.
The existence of the winning strategy for the duplicator now follows by induction.

It remains to consider the case that $w_{q+1}\in B\setminus B_W$.
Let $u$ be the column of $w_{q+1}$ in the adjacency matrix of $G$ restricted to $W_A$.
Clearly, $u$ is contained in the $2^{p}$-shadow of $W$.
Consequently,
there is a vertex $w'_{q+1}\in B'\setminus B'_{W'}$ such that the column of $w'_{q+1}$ in the adjacency matrix of $G'$ restricted to $W'_{A'}$ is $u$.
The duplicator responds with the vertex $w'_{q+1}$.
The adjacency matrices of $G$ and $G'$ restricted to $W\cup\{w_{q+1}\}$ and $W'\cup\{w'_{q+1}\}$, respectively, are the same.
The $2^{p-1}$-shadow of $W\cup\{w_{q+1}\}$ in $G$ is obtained from the $2^{p}$-shadow of $W$ by removing $u$ from the shadow and
restricting the multiplicity of each vector to be at most $2^{p-1}$.
Likewise,
the $2^{p-1}$-shadow of $W'\cup\{w'_{q+1}\}$ in $G'$ is obtained from the $2^{p}$-shadow of $W'$ by removing $u$ from the shadow and
restricting the multiplicity of each vector to be at most $2^{p-1}$.
Note that if $A=A'=\emptyset$, each of the $2^{p-1}$-shadows consists of $2^{p-1}$ null vectors.
Since the $2^{p}$-shadows of $W$ and $W'$ are the same, the $2^{p-1}$-shadows of $W\cup\{w_{q+1}\}$ and $W'\cup\{w'_{q+1}\}$ are also the same.
The existence of the winning strategy for the duplicator now follows by induction.
\end{proof}

Let $\ss=(s_n)_{n\in\NN}$ be a sequence of integers such that $s_n\ge 2$ for every $n\in\NN$.
For each $x\in [0,1]$, there exists a unique sequence $(x_n)_{n\in\NN}$ of integers such that
$$x=\sum_{n=1}^\infty \frac{x_n}{\prod_{k=1}^n s_k}\;\mbox{,}$$
$0\le x_n<s_n$ for every $n$, and there is no $n\in\NN$ such that $x_n\not=s_n$ and $x_{n'}=s_{n'}$ for every $n'\ge n$.
We define $M_{\ss}$ to be the following modeling.
The vertex set of $M_{\ss}$ is the unit square $[0,1]^2$ with the uniform measure on its Borel subsets.
Fix a sequence $(z_n)_{n\in\NN}$ of distinct vertices, say $z=(2^{-n},0)$, and let $Z=\{z_n,n\in\NN\}$.
The modeling $M_{\ss}(A,B)$ is the bipartite graph with $A=[0,1]^2\setminus Z$ and $B=Z$ such that
a vertex $(x,y)\in A=[0,1]^2\setminus Z$ is adjacent to a vertex $z_n\in B=Z$ if and only if $x_n\not=0$.

We next verify that every first order definable subset of $M_{\ss}^k$ is Borel.
A subset $X$ of $M_{\ss}^\ell$ is {\em basic}
if there exist $v_1,\ldots,v_p\in B$ (we allow $p=0$),
a matrix $M\in\{0,1\}^{\ell\times p}$, an integer $q$ and a multiset $S\subseteq\{0,1\}^p$ such that
the set $X$ is formed by all $\ell$-tuples $w_1,\ldots,w_\ell\in A$ such that
the adjacency matrix restricted to $\{v_1,\ldots,v_p,w_1,\ldots,w_\ell\}$ is $M$ and
the $2^q$-shadow of $\{v_1,\ldots,v_p,w_1,\ldots,w_\ell\}$ is $S$.
In particular, if $X$ is basic, then $X\subseteq A^\ell$.

Let $X(\ell,B',M,T)$
for a non-negative integer $\ell$, a finite subset $B'\subseteq B$, a matrix $M\in\{0,1\}^{\ell\times B'}$ and a subset $T\subseteq\{0,1\}^{\ell}$,
be the set of $\ell$-tuples $w_1,\ldots,w_\ell\in A$ such that
the the adjacency matrix of $M_{\ss}$ restricted to $\{w_1,\ldots,w_\ell\}\times B'$ is $M$ and
all the columns of the adjacency matrix not associated with vertices of $B'$ belong to $T$ after restricting to $w_1,\ldots,w_\ell$.
Observe that the set $X(\ell,B',M,T)\subseteq A^\ell$ is Borel for all $\ell$, $B'$, $M$ and $T$.
Since every basic set is a countable union of sets $X(\ell,B',M,T)$, every basic set is Borel.

Fix a first order formula $\psi$ with $k$ free variables and quantifier depth $d$.
By Lemma~\ref{lm-key},
the set of $k$-tuples of $M_{\ss}^k$ satisfying $\psi$ can be partitioned into countably many subsets such that
each of them after a suitable permutation of coordinates
is either a basic set or a product of a basic set and one or more single element subsets of $B$.
Consequently, every first order definable subset of $M_{\ss}^k$ is Borel.

The next lemma directly follows from the definition of a modeling $M_{\ss}$.

\begin{lemma}
\label{lm-modeling}
Let $\ss=(s_n)_{n\in\NN}$ be a sequence of integers such that $s_n\ge 2$ for every $n\in\NN$.
The modeling $M_{\ss}(A,B)$ is $\ell$-universal for every $\ell\in\NN$.
For all integers $p$ and $\ell$, it holds with probability one that
a random $p$-tuple of vertices of $M_{\ss}$ contains $p$ different vertices from $A$ and
the $\ell$-shadow of the $p$-tuple is the multiset containing each vector $\{0,1\}^p$ with multiplicity $\ell$.
\end{lemma}

Observe that Lemmas~\ref{lm-key} and~\ref{lm-modeling} yield that
$\langle\psi,M_{\ss}\rangle=\langle\psi,M_{\ss'}\rangle$ for every first order formula $\psi$ and
any two sequences $\ss$ and $\ss'$.

We now define the graph $H_n(A,B)$ to be the graph with $A=[2^n]\times [n]$ and $B=[n]$ such that
$(a,a')\in A$ is adjacent to $b\in B$ iff the $b$-th bit of $a$ when written in binary is $1$.
We summarize the properties of the graphs $H_n$ in the next lemma.

\begin{lemma}
\label{lm-graphs}
Let $p$ and $\ell$ be two integers.
The graph $H_n(A,B)$ is $\ell$-universal if $n\ge\ell$, and
the probability that a random $p$-tuple of vertices of $H_n$ contains
$p$ different vertices from $A$ and the $\ell$-shadow of them
is the multiset containing each vector $\{0,1\}^p$ with multiplicity $\ell$
tends to one as $n$ tends to infinity.
\end{lemma}

The next theorem follows directly from Lemmas~\ref{lm-key}--\ref{lm-graphs}.

\begin{theorem}
\label{lm-limit}
Let $\ss=(s_n)_{n\in\NN}$ be a sequence of integers such that $s_n\ge 2$ for every $n\in\NN$.
It holds for every first order formula $\psi$ that
$$\lim_{n\to\infty}\langle\psi,H_n\rangle = \langle\psi,M_{\ss}\rangle\;\mbox{.}$$
In particular, the modeling $M_{\ss}$ is a limit of $(H_n)_{n\in\NN}$.
\end{theorem}

Let $\psi_0(x)$ be the first order formula that is true if $x$ is adjacent to the root.
If $s_n=3$ for every $n\in\NN$,
then the set of neighbors of every vertex of $B$ in the modeling $M_{\ss}(A,B)$ has measure $2/3$. 
Hence,
if $M'_{\ss}$ is the modeling obtained from $M_{\ss}$ by rooting at an arbitrary vertex of $B$,
then $\langle\psi_0,M'_{\ss}\rangle=2/3$.
Since no vertex of $H_n$ is adjacent to more than $2^{n-1}n$ vertices (the vertices of $A$
are adjacent to at most $n$ vertices each and each vertex of $B$ is adjacent to $2^{n-1}n$),
it holds that
$$\limsup_{n\to\infty}\langle\psi_0,H'_n\rangle\le\frac{1}{2}$$
for any sequence $(H'_n)_{n\in\NN}$ of rooted graphs obtained from $H_n$.
We conclude that the sequence $(H_n)_{n\in\NN}$, the modeling $M_{\ss}(A,B)$ with $\ss=(3)_{n\in\NN}$ and
rooting $M_{\ss}$ at any vertex of $B$ provide a counterexample to Problem~\ref{prob-1}.

\end{document}